\documentclass[11pt,a4paper]{article}

\title{Free subgroup of the $C^0$ symplectic mapping class group}
\author{Alexandre Jannaud }
\date{November 2022}

\usepackage[utf8]{inputenc}
\usepackage[T1]{fontenc}

\usepackage[a4paper]{geometry}
\usepackage[french,english]{babel}

\usepackage{amsthm}

\usepackage{fancyhdr}
\usepackage{import}
\usepackage[final]{graphicx}
\usepackage{numprint} 
\usepackage{adjustbox}
\usepackage{varioref}
\usepackage{hyperref}
\usepackage{nameref}

\usepackage{caption}
\usepackage{yfonts}
\usepackage{thmtools}

\usepackage{setspace}
\usepackage{amsmath,amssymb,latexsym,amsthm,epsfig,stmaryrd}
\usepackage{tikz-cd}
\usepackage{mathrsfs}
\usepackage[all]{xy} 
\usepackage{multirow} 

\usepackage[capitalize,nameinlink,noabbrev]{cleveref}

\allowdisplaybreaks[1]

\usepackage{calc}

\newcommand{\R}{\mathbb{R}}

\newcommand{\Z}{\mathbb{Z}}

\newcommand{\N}{\mathbb{N}}

\newcommand{\Symp}{\mathrm{Symp}}

\DeclareRobustCommand\openone{\leavevmode\hbox{\small1\normalsize\kern-.33em1}}

\def\calA{\mathcal{A}}
\def\calB{\mathcal{B}}

\newtheorem{thm}{Theorem}[section]
\newtheorem{lemma}[thm]{Lemma}
\newtheorem{prop}[thm]{Proposition}
\newtheorem{dfn}[thm]{Definition}

\newtheorem{thmx}{Theorem}
\crefname{thmx}{Theorem}{Theorems}

\crefname{corx}{Corollary}{Corollaries}

\theoremstyle{definition}

\theoremstyle{definition}

\begin{document}

\maketitle

\begin{abstract}
Using the technology of barcodes and previously proven continuity results, we extend to $C^0$ symplectic topology a beautiful result from Keating [`Dehn twists and free subgroups of symplectic mapping class groups', Journal of Topology 7 (2014) 436-474]. Given two Lagrangian spheres in a Liouville domain, with good conditions, we prove that the Dehn twists about these spheres generate a free subgroup of the $C^0$ symplectic mapping class group.
\end{abstract}
\setcounter{tocdepth}{1}
\tableofcontents
\section*{Introduction}\label{introduction}

\addcontentsline{toc}{section}{Introduction}

We will work with a Liouville domain $(M^{2n},\omega)$, together with Lagrangian spheres $L$ and $L'$. The group of symplectomorphisms will be denoted $\Symp(M,\omega)$. We are here interested in $C^0$ symplectic topology. This field was born after the famous Gromov-Eliashberg theorem \cite{El87} stating that if a sequence of symplectomorphisms $C^0$-converges to a diffeomorphism, then this diffeomorphism is a symplectomorphism as well. The objects that interest us in this paper are symplectic homeomorphisms.
\begin{dfn}
Let $(M,\omega)$ be a symplectic manifold equipped with a Riemannian metric. A homeomorphism $\varphi$ of $M$ is called a \emph{symplectic homeomorphism} if it is the uniform limit of a sequence of symplectic diffeomorphisms.
\end{dfn}
We will denote $\overline{\Symp}(M,\omega)$ the set of symplectic homeomorphisms, which is equipped with the $C^0$-topology.
Let us introduce the following natural map induced by the inclusion of $\Symp(M,\omega)$ in $\overline{\Symp}(M,\omega)$
\begin{equation}\label{formule J}
\pi_0(\Symp(M,\omega))\overset{J}\longrightarrow \pi_0(\overline{\Symp}(M,\omega)).    
\end{equation}
This goal for this paper is to continue the study of the $C^0$symplectic mapping class group, initiated in \cite{J21}, through this map $J$. For further motivation regarding the study of this map, the reader can refer to the afford mentioned article.
In \cite{J21} we proved that this map $J$, in some cases, is not trivial, more precisely that it sends a non-zero element in $\pi_0(\Symp(M,\omega))$ onto a non-zero element in $\pi_0(\overline{\Symp}(M,\omega))$. Here, we will prove the following more subtle theorem, stating that, with good assumptions, $J$ sends a free subgroup of $\pi_0(\Symp(M,\omega))$ onto a free subgroup of $\pi_0(\overline{\Symp}(M,\omega))$.
\begin{thmx}\label{thm free subgroup}
Let $n\geq2$, $(M^{2n},\omega)$ a Liouville domain and $L,L'$ two Lagrangian spheres. We ask that $\mathrm{dim}(HF(L,L'))\geq 2$, and if it is equal to $2$, that $L$ and $L'$ are not isomorphic in $Fuk(M)$. Then, the Dehn twists $\tau_L$ and $\tau_{L'}$ generate a free subgroup of $\pi_0(\overline{\Symp}(M,\omega))$. 
\end{thmx}
Here, $\tau_L$ denotes the generalized Dehn twist about this Lagrangian sphere $L$. This Dehn twist a symplectomorphism supported in the neighbourhood of $L$, and if $M$ is exact, then so is $\tau_L$. It is a generalization, by Arnold \cite{Arn95}, to higher dimensions of the classical Dehn twist on surfaces, and so are of particular interest when studying a given mapping class group.

This theorem is the extension to the $C^0$ setting of the exact same result by Keating \cite{Kea12} in the smooth case. Our theorem is consequently a rigidity result for $C^0$symplectic topology. This can come as a bit surprising since, a priori, the topology of $\overline{\Symp}(M,\omega)$ could be less rich than the one of $\Symp(M,\omega)$. Indeed the first one is bigger than the second one and is equipped with a less restrictive topology.

The assumptions on the Lagrangian submanifolds in our theorem are exactly the same than the ones involved in Keating's work. As we deeply need her results and computations to prove our result, we ask for the same properties regarding the Lagrangian submanifolds. Such configurations can easily be found, for example, in Milnor fibers of type $A_2$ singularities. One can indeed construct three Lagrangian spheres $L_0,L_1,L_2$, such that $dim(HF(L_0,L_1))=2$, $dim(HF(L_0,L_2))=1$ and $dim(HF(L_1,L_2))=3$ (\cite{KovSeid00}). Configurations satisfying our theorem's hypothesis also exists in many other examples; these are detailed in the last section of Keating's paper \cite{Kea12}, so we will not discuss it any further.

Keating's results relies on inequalities concerning the rank of the Floer cohomology when we apply the Dehn twist to one of the Lagrangian submanifolds. The starting point of her proofs of these inequalities is Seidel's long exact sequence for the generalized Dehn twist \cite{Seidel01}.

The other tool used in our paper is the technology of barcodes. These come from topological data analysis, with the work of Edelsbrunner et al. \cite{Edel&al00}, Carlson et al. \cite{Carlson&al04}. The terminology of barcodes was brought into symplectic topology by Polterovich and Shelukhin \cite{PolShel14} although germs of this theory were already present in the work of Barannikov \cite{Bar94} and Usher \cite{Ush11,Ush14}. In this case, they are constructed from Floer homology, and they can be understood as a "patch" of spectral invariants. The space of barcodes can be equipped with a distance, called the the \emph{bottleneck distance}. In a previous paper (\cite{J21}), we proved a $C^0$-continuity result for barcodes in symplectic topology.

The proof of \autoref{thm free subgroup} consists in putting together the author previous result \cite{J21} on the $C^0$-continuity of barcodes together with Keating's results and computations \cite{Kea12}. 

\subsection*{Organisation}
In the first section, we will briefly recall, for the sake of notations, the object involved in Floer cohomology, Keating's results and finally the technology of barcodes and the properties and theorems we will need. The second section is dedicated to the proof of \autoref{thm free subgroup}.

\subsection*{Acknowledgments}

The author is utterly grateful to Ailsa Keating for her suggestion of this result and the discussions she was kind enough to have about her paper. He is also thankful to his PhD advisor Vincent Humilière for his precious remarks. This paper is funded by the Deutsche Forschungsgemeinschaft (DFG, German Research Foundation) under Germany's Excellence Strategy EXC 2181/1 - 390900948 (the Heidelberg STRUCTURES Excellence Cluster).

\section{Preliminaries}\label{sec preli}
\subsection{Floer cohomology}
All the following notions in this section are originally due to Floer \cite{Floer88}. As these are classical objects, we just want to introduce some conventions and notations and thus will be brief. For more details, one can refer to  \cite{Sei08}.

Let $(M,\omega)$ be a Liouville domain, with $d\lambda=\omega$, and let $L$ and $L'$ be two closed connected exact Lagrangian submanifolds in $M$, together with $f_L:L\rightarrow \R$ and $f_{L'}:L'\rightarrow \R$ such that $df_L=\lambda_{|L}$ and $df_{L'}=\lambda_{|L'}$. To achieve the transversality and compactness, we consider a generic choice Hamiltonian and (time dependant) almost-complex structure perturbations, which we will denote by the pair $(H,J_t)$ or simply $(H,J)$. This generic choice ensures that all the Floer cohomology involved in this paper are well-defined. The generators $\chi_H(L,L')$ of this Floer complex are then the flow lines $\gamma:[0;1]\rightarrow M$ such that $\dot{\gamma}(t)=X_H(t,\gamma(t)),$ and $\gamma(0)\in L,\gamma(1)\in L'.$

The Hamiltonian action of a path $\gamma$ from $L$ to $L'$ is then defined as
\begin{equation}\label{formule action hamiltonienne}
    \calA_{L,L'}^H(\gamma)=\int_0^1\gamma^*\lambda-H(\gamma)dt+f_L(\gamma(0))-f_{L'}(\gamma(1)).
\end{equation}
The critical points of this action are the above mentioned generators of the Floer complex. The differential of this Floer complex is defined by counting pseudo-holomorphic strips between these critical points. In the end, we get the Floer complex $CF(L,L';H,J),$
and then the Floer cohomology $HF(L,L';H,J)$.

The action gives rise to a filtration on the Floer complex. For all $\kappa\in\R$, we define
$$CF^{\kappa}(L,L';J_t,H)=\mathrm{span}_{\Z/2}\left\{z\in\chi(L,L'),~~\calA^H_{L,L'}(z)<\kappa\right\}\subset CF^(L,L';J_t,H).$$
We can then compute the cohomology of this subcomplex $CF^{*}(L,L';J_t,H)$, and consequently define:
$$HF^{\kappa}(L,L';J_t,H)=H(CF^{\kappa}(L,L';J_t,H)).$$

Up to canonical isomorphism, the group of the Floer cohomology is independent of the choice of the pair $(H,J)$. Thus, we will simply denote $HF(L,L')$ the Lagrangian Floer cohomology complex associated to the Lagrangian submanifolds $L$ and $L'$. Using Keating's notations \cite{Kea12}, we will denote $hf(L,L')$ the rank of the complex $HF(L,L')$.

We recall that the \emph{Fukaya category} of $M$, $Fuk(M)$, is an $A_{\infty}-$category whose objects are the Lagrangian submanifolds of $M$ and whose morphisms are the Floer chain groups between the Lagrangian submanifolds. For more details, one can refer to \cite{Aur14,Fuk93}.

Finally, we will be working with Lagrangian spheres which are automatically exact.

\subsection{Dehn twists and Keating's results}

Let $L$ be a Lagrangian sphere in $M$. We will denote $\tau_L$ the generalized Dehn twist about this Lagrangian sphere $L$. In our context, this map is an exact symplectomorphism, and it is supported in a neighbourhood of $L$ and its restriction to $L$ is the antipodal map. More details on these maps can be found in \cite{Arn95,KK05,KRW21,KovSeid00,Schlenk18,ST01,Sei08}.

In \cite{Kea12}, Keating proved the existence of a free subgroup in the symplectic mapping class group. To do so, she proved and used the following lemmas, which will be key arguments in our extension to a $C^0$-setting. Both these lemmas give inequalities on the rank of the Floer cohomology under the action of the generalized Dehn twist.
\begin{lemma} \emph{(see \cite{Kea12}, Lemma $8.1$)}\label{Keating iter}
Let $n\geq 2$ and $M^{2n}$ be an exact symplectic manifold with contact type boundary. Let $L,L_0, L_1$ be Lagrangian submanifolds such that $L$ is a sphere, not quasi-isomorphic to $L_0$ in $Fuk(M)$, and $hf(L,L_0)\geq 2$. Then for all $k\in\Z^*$,
$$hf(L,L_1)>hf(L_0,L_1)\Rightarrow hf(L,\tau_L^k(L_1))<hf(L_0,\tau_L^k(L_1)).$$
\end{lemma}

\begin{lemma}\emph{(see \cite{Kea12}, Claim $8.2$)}\label{Keating init}
Let $n\geq2$ and $M^{2n}$ an exact symplectic manifold with contact type boundary. Let $L,L'$ two Lagrangian spheres such that $hf(L,L')\geq 2$, and if $hf(L,L')=2$, we ask that $L$ and $L'$ are not quasi-isomorphic in $Fuk(M)$.Then, for all $k\in\Z^*$, we have $hf(L',\tau_{L'}^kL)<hf(L,\tau_{L'}^kL)$.
\end{lemma}

\subsection{Barcodes and Lagrangian Floer cohomology}

\begin{dfn}
A \emph{barcode} $B$ is a multiset of non-empty intervals, called \emph{bars,} in $\R$ of the form $(a,b]$ or $(a,+\infty)$, with $a$, $b$ in $\R$.
\end{dfn}

We can equip the set of barcodes with a distance, which is called the \emph{bottleneck distance}. Let $I$ be a non-empty interval of the form $(a,b]$ or $(a,+\infty)$, and $\delta\in\R$ such that $2\delta<b-a$. We denote $I^{-\delta}$ the interval $(a-\delta,b+\delta]$ or $(a-\delta,+\infty)$.
\begin{dfn}
Let $B$ and $B'$ be two barcodes, and $\delta\geq 0$. They admit a $\delta$-matching if we can delete in both of them some bars of length smaller than $2\delta$ to get two barcodes $\bar{B}$ and $\bar{B}'$ and find a bijection $\phi:\bar{B}\rightarrow \bar{B}'$ such that if $\phi(I)=J$, then
$I\subset J^{-\delta} \text{ and } J\subset I^{-\delta}.$
\end{dfn}
The bottleneck-distance between the barcodes $B$ and $B'$ is then
$$d_{bottle}(B,B')= \inf\{\delta |~~B \text{ and }B' \text{ admit a $\delta$-matching}\}.$$
The bottleneck distance is non-degenerate.

Given a barcode $B$ and $\delta \in\R$, we will denote $B[\delta]$ the barcode obtained from $B$ by an overall shift of $\delta$, and $\calB$ the set of barcodes such that for all $\varepsilon>0$, the number of bars of length greater or equal to $\varepsilon$ is finite.
\begin{dfn}\label{def barcodes overall shift}
We define $\hat{\calB}$ as the set of barcodes $\calB$ quotiented by the action by overall shift of $\R$ on $\calB$, i.e. $B$ and $B'$ represent the same class in $\hat{\calB}$ if and only if there is $c\in\R$ such that $B=B'[c]$.
\end{dfn}
This space of barcodes $\hat{\calB}$ satisfies the following property, essential to take limits of sequences of barcodes as we do in $C^0$ symplectic topology.
\begin{lemma}[\cite{BV18}]
The set $\hat{\calB}$ is complete for the distance $\delta$.
\end{lemma}

It is easy to see that the connected components of $\hat{\calB}$ are indexed by the number of semi-infinite bars. In other words, two barcodes belong to the same connected component of $\hat{\calB}$ if and only if they have the same number of semi-infinite bars. Moreover the connected components are path-connected.
The following proposition directly follows from this fact, but its formulation will be useful later.

\begin{prop}\label{prop path inf bar}
Let $(B^t)_{t\in[0;1]}$ be a continuous path of graded barcodes. Then for all $t\in[0;1]$ and for all $k$, the number of semi-infinite bars of $B_k^t$ is constant with respect to the parameter $t$.
\end{prop}
We denote $\sigma^\infty:\hat{\calB}\rightarrow\N$ the map which to a barcode $B\in\hat{\calB}$ associate its number of semi-infinite bars.

Since the works \cite{Bar94,UZ15,PolShel14}, it is standard to associate a barcode to a Floer complex and we will not recall it here. This is done thanks to the filtration by the action (\ref{formule action hamiltonienne}). The end points of the bars correspond to birth and death of classes in the filtered Floer complex $HF^{*,\kappa}(L,L';J_t,H)$, as the filtration $\kappa$ runs over $\R$. A precise description of this association in our context can be found in \cite{J21}.
We will denote $\hat{B}(L,L';J_t,H)$ the barcode in $\hat{\calB}$ associated to the Floer complex $CF(L,L';J_t,H)$. We have the following equality:
\begin{eqnarray}
hf(L,L';J_t,H)=\sigma^{\infty}(\hat{B}(L,L';J_t,H)).
\end{eqnarray}

From now on, we will only care about the number of semi-infinite bars in the barcodes, so we will simply denote the barcodes associated to the complex $CF(L,L';J_t,H)$ by $\hat{B}(L,L')$.

Let us now recall the following theorems (\cite{J21}) which will be a key argument in our proof.
\begin{thm}\label{prop contL barcode}\emph{(see \cite{J21}, Theorem $3.1$)}
Let $M$ be a Liouville domain. Let $L$ and $L'$ be two closed exact Lagrangian submanifolds, and assume that $H^1(L',\R)=0$. Then the map $\varphi \mapsto \hat{B}(\varphi(L'),L)$ from $\mathrm{Symp}(M,\omega)$ to $\hat{B}$ is locally $C^0$-Lipschitz. Thus it continuously extends to a map $\overline{\mathrm{Symp}}(M,\omega)\rightarrow\hat{\calB}$.
\end{thm}

From \autoref{prop contL barcode}, we immediately obtained the useful following theorem. 
\begin{thm}\label{thm barcodes connected}\emph{(see \cite{J21}, Theorem 3.3)}
Let $M$ be a Liouville domain. Let $L$ and $L'$ be two exact compact Lagrangian submanifolds, and assume that $H^1(L',\R)=0$. Consider two symplectomorphisms $\varphi$ and $\psi$.
\begin{itemize}
    \item If these two symplectomorphisms are in the same connected component of $\overline{\Symp}(M,\omega)$, then the two barcodes $\hat{B}(\varphi(L'),L)$ and $\hat{B}(\psi(L'),L)$ are in the same connected component of $\hat{\calB}$.
    \item If these two symplectomorphisms are isotopic in $\overline{\Symp}(M,\omega)$, then there is a continuous path of barcodes from $\hat{B}(\varphi(L'),L)$ to $\hat{B}(\psi(L'),L)$.
\end{itemize}
\end{thm}

\section{Proof of \autoref{thm free subgroup}}\label{sec free subgroup}
This proof almost exactly follows Keating's, in a slightly different context. We need here to work with barcodes as Floer cohomology is not well defined for $C^0$ symplectic topology.

Suppose $n\geq2$. Let $(M^{2n},\omega)$ be an exact symplectic manifold with contact type boundary. Let $L$ and $L'$ be two Lagrangian spheres such that $dim(HF(L,L'))\geq 2$. Additionally, if there is equality, we ask for $L$ and $L'$ not to be quasi-isomorphic in the Fukaya category of $M$. We denote $\tau_L$ and $\tau_{L'}$ the two Dehn twists around $L$ and $L'$.

Let us first deal with the case $dim(HF(L,L'))> 2$.
\begin{lemma}\label{lem conn comp DT}
For all $k\in\Z^*$, the maps $\tau_L^k$ and $\tau_{L'}^k$ are not in the component of $Id$ in $\pi_0(\overline{Symp}(M,\omega))$.
\end{lemma}
\begin{proof}
This is a direct consequence of \autoref{Keating iter} together with \autoref{thm barcodes connected}. Assume that there exists $n_0\in\Z^*$ such that $\tau_L^{n_0}$ is isotopic to the identity in $\overline{\Symp}(M,\omega)$. Then there exist continuous paths $(B_t)_{t\in[0;1]}$ and $(B'_t)_{t\in[0;1]}$ of barcodes such that 
$$B'_0=B(L,L')\text{ and } B'_1=B(L,\tau_L^{n_0}L'),$$
$$B_0=B(L',L')\text{ and } B_1=B(L',\tau_L^{n_0}L').$$
Since $hf(L,L')> 2$, we have $\sigma^{\infty}(B'_0)>\sigma^{\infty}(B_0)$. For $L$ corresponding to $L$ and $L'=L_0=L_1$, \autoref{Keating iter} tells us that $\sigma^{\infty}(B'_1)<\sigma^{\infty}(B_1)$. But, according to \autoref{prop path inf bar}, the number of semi-infinite bars does not change along a continuous path of barcodes. So we get a contradiction and there exists no such $n_0$. The exact same argument applies for $\tau_{L'}$.
\end{proof}
So if the subgroup generated by $\tau_L$ and $\tau_{L'}$ is not free, then there exists $k\in\N$ and $(a_i,b_i)_{i\in\llbracket1,k\rrbracket}\in(\Z^*\times\Z^*)^k$ such that
$$\psi:=\tau_{L'}^{b_k}\tau_{L}^{a_k}\dots\tau_{L'}^{b_1}\tau_{L}^{a_1}=1\in\pi_0(\overline{Symp}(M,\omega)).$$
By the second point of \autoref{thm barcodes connected}, we can define  two paths of barcodes $(B'_t)_{t\in[0,1]}$ and $(B_t)_{t\in[0,1]}$ such that
$$B'_0=B(L',L)\text{ and } B'_1=B(L',\psi(L)),$$
$$B_0=B(L,L)\text{ and } B_1=B(L,\psi(L)).$$
These two paths are continuous paths of barcodes, so they keep the same number of semi-infinite bars at all times, i.e.
\begin{equation}\label{egalité nb barres}
    \sigma^{\infty}(B'_0)=\sigma^{\infty}(B'_1)\text{ and }\sigma^{\infty}(B_0)=\sigma^{\infty}(B_1).
\end{equation}

By iterating \autoref{Keating iter} to the initial inequality $hf(L',\tau_L^{a_1}l)>hf(L,\tau_L^{a_1}l)$, Keating gets that 
$$hf(L',\tau_{L'}^{b_k}\tau_{L}^{a_k}\dots\tau_{L'}^{b_1}\tau_{L}^{a_1}L)<hf(L,\tau_{L'}^{b_k}\tau_{L}^{a_k}\dots\tau_{L'}^{b_1}\tau_{L}^{a_1}L).$$
In terms of barcodes, this now translates as
$$\sigma^{\infty}(B(L',\psi(L)))<\sigma^{\infty}(B(L,\psi(L))),$$
$$\text{i.e., }\sigma^{\infty}(B'_1)<\sigma^{\infty}(B_1).$$
Then, by equality \ref{egalité nb barres}, we get that $\sigma^{\infty}(B'_0)<\sigma^{\infty}(B_0)$.
Nevertheless, since $L$ is a Lagrangian sphere, we have $\sigma^{\infty}(B_0)=2$ and thus $$\sigma^{\infty}(B'_0)>\sigma^{\infty}(B_0),$$
which leads us to a contradiction. So, $\tau_L$ and $\tau_{L'}$ generate a free subgroup.

We now turn to the case $hf(L,L')=2$. With a similar argument as for \autoref{lem conn comp DT} applied to the inequality of \autoref{Keating init}, we also get in this case that for all $k\in\Z^*$, the maps $\tau_L^k$ and $\tau_{L'}^k$ are not in the component of $Id$ in $\pi_0(\overline{Symp}(M,\omega))$. So, if the subgroup generated by $\tau_L$ and $\tau_{L'}$ is not free, then there exists $k\in\N$ and $(a_i,b_i)_{i\in\llbracket1,k\rrbracket}\in(\Z^*\times\Z^*)^k$ such that
$$\psi:=\tau_{L'}^{b_k}\tau_{L}^{a_k}\dots\tau_{L'}^{b_1}\tau_{L}^{a_1}=1\in\pi_0(\overline{Symp}(M,\omega))
.$$
With the same notations as before, we define two paths of barcodes $(B'_t)_{t\in[0,1]}$ and $(B_t)_{t\in[0,1]}$ such that
$$B'_0=B(L',L)\text{ and } B'_1=B(L',\psi(L)),$$
$$B_0=B(L,L)\text{ and } B_1=B(L,\psi(L)).$$

Now, \autoref{Keating init} also tells us that for all $a_1\in\Z$ and $b_1\in\Z^*$, $hf(L',\tau_{L'}^{b_1}\tau_L^{a_1}L)<hf(L',\tau_{L}^{b_1}\tau_L^{a_1}L)$. 
As before, using \autoref{Keating iter}, Keating gets that
$$hf(L',\tau_{L'}^{b_k}\tau_{L}^{a_k}\dots\tau_{L'}^{b_1}\tau_{L}^{a_1}L)<hf(L,\tau_{L'}^{b_k}\tau_{L}^{a_k}\dots\tau_{L'}^{b_1}\tau_{L}^{a_1}L),$$
$$\text{i.e., }\sigma^{\infty}(B(L',\psi(L)))<\sigma^{\infty}(B(L,\psi(L))),$$
$$\text{i.e., }\sigma^{\infty}(B'_1)<\sigma^{\infty}(B_1).$$
So, by equality \ref{egalité nb barres}, we get
$$\sigma^{\infty}(B'_0)<\sigma^{\infty}(B_0),$$
which is a contradiction since we actually have
$$\sigma^{\infty}(B'_0)=hf(L,L')=2=hf(L,L)=\sigma^{\infty}(B_0).$$
 So $\tau_L$ and $\tau_{L'}$ generate a free subgroup.

\bibliographystyle{siam}
\bibliography{Morceaux/Bibli}

\end{document}